\documentclass[a4paper, 12pt]{article}
\usepackage{amsmath,amsthm, amssymb, esint}
\usepackage{graphicx}
\usepackage{graphics}
\usepackage{marginnote}
\usepackage{framed}
\usepackage{color}
\definecolor{red1}{rgb}{1.00,0.00,0.00}

\newcommand{\Addresses}{{
  \bigskip
  \footnotesize

  \textsc{Department of Mathematics, Saarland University, P.O. Box 151150,  Saar- br{\"u}cken 66041, Germany} and
  \textsc{ Faculty of Mathematics and Mechanics, St. Petersburg State University, Universitetskii pr. 28,  St. Petersburg 198504, Russia}\par\nopagebreak
  \textit{E-mail address:} \texttt{darya@math.uni-sb.de}

  \bigskip

  \textsc{Faculty of Mathematics and Mechanics, St. Petersburg State University, Universitetskii pr. 28,  St. Petersburg 198504, Russia}\par\nopagebreak
  \textit{E-mail address:} \texttt{uraltsev@pdmi.ras.ru}
}}

\theoremstyle{definition}
\theoremstyle{plain}
\newtheorem{thm}{Theorem}[section]
\newtheorem{lemma}[thm]{Lemma}
\newtheorem{cor}[thm]{Corollary}
\newtheorem{fact}[thm]{Fact}
\newtheorem{remark}[thm]{Remark}
\title{On regularity properties of solutions to the hysteresis-type problem
\thanks{This work was supported by the Russian Foundation of Basic Research (RFBR) through the grant number
14-01-00534, and by the Thematic Plan of the St. Petersburg State University, and by the St. Petersburg State University grant 6.38.670.2013.}
}
\author{D.E.\,Apushkinskaya and N.N.\,Uraltseva} 

\begin{document}
\maketitle

\begin{abstract}
We consider equations with the simplest hysteresis operator at the right-hand side. Such equations describe the so-called processes "with memory" in which various substances interact according to the hysteresis law.

We restrict our consideration on the so-called "strong solutions" belonging to the Sobolev class $W^{2,1}_q$ with sufficiently large $q$ and prove that in fact $q=\infty$. In other words, we establish the optimal regularity of solutions. Our arguments are based on quadratic growth estimates for solutions near the free boundary. 
\end{abstract}

\bibliographystyle{alpha}
\section{Introduction.}

In this paper we study the regularity properties of  solutions of the following parabolic equation:
\begin{equation} \label{main-equation}
H[u]=h[u] \quad \text{in}\quad Q=\mathcal{U}\times ]0,T].
\end{equation}
Here $H=\Delta-\partial_t$ is the heat operator, $\mathcal{U}$ is a domain in $\mathbb{R}^n$, and $h$ is a hysteresis-type operator acting from $C(\overline{Q})$ to $\left\lbrace \pm 1\right\rbrace $ which is defined as follows.

We fix two numbers $\alpha$ and $\beta$ ($\alpha <\beta$) and consider a \textbf{multivalued} function
$$
f(s)=\left\lbrace 
\begin{aligned}
-&1,  \quad \text{for}\quad s\in ]-\infty, \alpha],\\
&1,  \quad \text{for}\quad s\in [\beta, +\infty[,\\
-&1\ \text{or}\ 1,  \quad \text{for} \quad s \in ]\alpha, \beta[.
\end{aligned}
\right.
$$

For $u\in C(\overline{Q})$ we suppose that on the bottom of the cylinder $Q$ the initial values of $u$ as well as of $h[u](x,0):=f(u(x,0))$  are prescribed. 

After that for every point $z=(x,t)\in Q$ the corresponding value of $h[u](z)$ is uniquely defined in the following manner. Let us denote by $E$ a set of points
$$
E:=\left\lbrace z\in Q : u(z)\leqslant \alpha\right\rbrace \cup \left\lbrace z\in Q : u(z) \geqslant \beta\right\rbrace
\cup \left\lbrace \mathcal{U} \times \left\lbrace 0\right\rbrace \right\rbrace  .
$$
In other words, $E$ is a set where $f(u(z))$ is well-defined.

If $z\in E$ then $h[u](z)=f(u(z))$. Otherwise, for $z=(x,t)\in Q$ such that $\alpha <u(z)<\beta$ we set
\begin{equation} \label{definition-H}
h[u](x,t)=h[u](x, \hat{t} (x)).
\end{equation}
Here
$$
\hat{t} (x)=\max\left\lbrace s : (x,s)\in E;\ s \leqslant t\right\rbrace 
$$

Roughly speaking, condition (\ref{definition-H}) means that the hysteresis function $h[u](x,t)$ takes for $u(x,t)\in (\alpha, \beta)$ the same value as  "at the previous moment" (see Figure~\ref{bild-h-operator}).

\begin{figure}[htbp]
\centering
\includegraphics[width=0.85\textwidth]{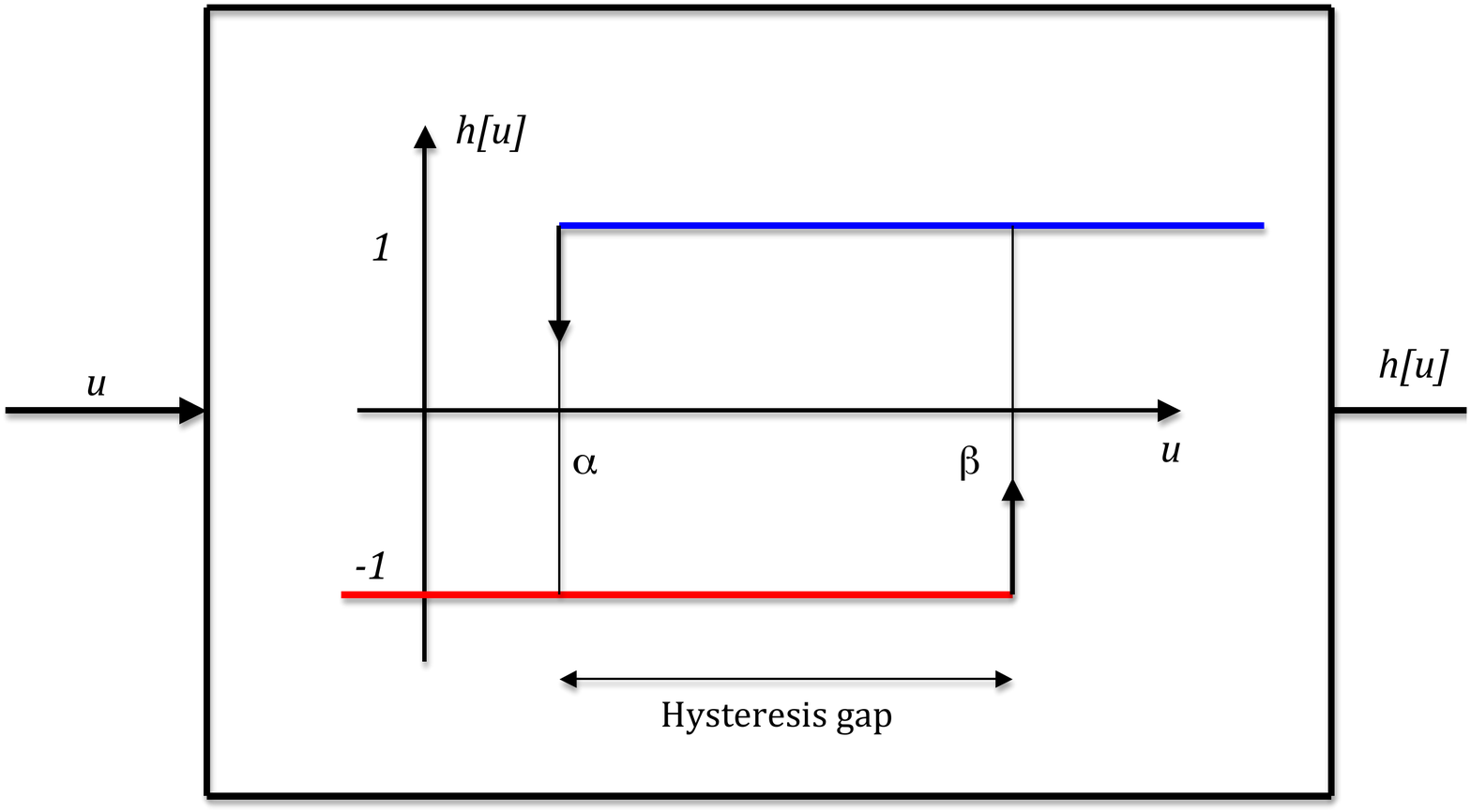}
\caption{The hysteresis operator $h$}
\label{bild-h-operator}
\end{figure}
Let us emphasize that for fixed $x$ a jump of $h[u](x,\cdot)$ can happen only on thresholds $\left\lbrace u(x,t) =\alpha\right\rbrace $ and $\left\lbrace u(x,t)=\beta\right\rbrace $. Moreover, \textbf{"jump down"} (from $h=1$ to $h=-1$)  is possble on  $\left\lbrace u(x,t)=\alpha\right\rbrace $ only, whereas \textbf{"jump up"} (from $h=-1$ to $h=1$)  is possible on $\left\lbrace u(x,t)=\beta\right\rbrace$ only. 

We say that $u$ is a (strong) solution of Eq. (\ref{main-equation}) if $u\in W^{2,1}_q(Q)$, $q>n+2$, and $u$ satisfies (\ref{main-equation}) a.e. in $Q$. In particular, it implies that the $(n+1)$-dimensional Lebesgue measure of the sets $\left\lbrace u=\alpha\right\rbrace $ and $\left\lbrace u=\beta\right\rbrace $ equals zero.

Thus, the cylinder $Q$ consists of two disjoint regions where $h[u]$ assumes the values $+1$ and $-1$, respectively. If $u$ is a solution of (\ref{main-equation}) then the interface between these two regions is apriori unknown and, therefore, may be considered as the free boundary.

\vspace{0.3cm}

Equation of type (\ref{main-equation}) arises in various biological and chemical processes in which diffusive and nondiffusive substances interact according to hysteresis law (see, for instance, \cite{HJ80}, \cite{HJP84}, \cite{K06}, and references therein). 

Difficulties in study of challenging hysteresis phenomenon include the discontinuous nonlinearity and the multivalence of  corresponding operator as well. A first attempt to create a mathematical theory of hysteresis was made in the  monograph \cite{KrP83}. We mention also the fundamental books \cite{V94}, \cite{BS96} and \cite{Kr96} where the hysteretic effects in spatial-distributed systems are described. The above-listed monographs are mainly devoted to the existence results as well as to investigations of qualitative properties of solutions.

The solvability of initial-boundary value problems for equation (\ref{main-equation}) was studied in papers \cite{Al85} and \cite{V86} in one-(space)-dimensional case and in multi-(space)-dimensional case, respectively. The global existence in a specially defined classes of weak solutions were established there. Moreover, in \cite{Al85} the nonuniqueness and nonstability of such weak solutions were discussed in several examples.
Recently, in papers 
\cite{GuShTi13} and \cite{GuTi12} the strong transversal solutions, belonging to the Sobolev space $W^{2,1}_q$ with suffiently large $q$, were studied in the one-(space)-dimensional case. This transversality property roughly speaking means that the solution has a nonvanishing spatial gradient on the free boundary. In the paper \cite{GuShTi13} the authors proved the local existence of strong transversal solutions  and showed that such solutions depend continuously on initial data. A theorem on the uniqueness of strong transversal solutions was established in \cite{GuTi12}.

\vspace{0.3cm}

In this paper we are interested in local $L^{\infty}$-estimates for the derivatives $D^2u$ and $\partial_tu$ of the strong solutions  of Eq. (\ref{main-equation}). We do not suppose that our solutions have the transversality property.
\vspace{0.2cm}

We assume that 
\begin{equation} \label{sup-estimate}
\sup\limits_{Q}|u| \leqslant M \quad \text{with}\quad  M > 1.
\end{equation}

 Since the right-hand side of (\ref{main-equation}) is bounded,  the general parabolic theory (see, e.g. \cite{LSU67}) implies for any $\epsilon>0$ the estimates
\begin{equation} \label{W^2_q-estimates}
\|\partial_t u\|_{q, Q^{\epsilon}}+\|D^2u\|_{q, Q^{\epsilon}}\leqslant N_1(\epsilon, q,M) \quad \forall q<\infty, 
\end{equation}
where $Q^{\epsilon}=\mathcal{U}^{\epsilon} \times ]\epsilon^2, T]$, $\mathcal{U}^{\epsilon} \subset \mathcal{U}$ and $\textit{dist}\, \left\lbrace \mathcal{U}^{\epsilon}, \partial\mathcal{U}\right\rbrace \geqslant \epsilon$. 
\vspace{0.2cm}

In particular, (\ref{W^2_q-estimates}) implies that functions $u$ and $Du$ are H{\"o}lder continuous in $Q$. 

We note that 
if  $\partial\mathcal{U}$ as well as the values of $u$ on the parabolic boundary of $Q$ are smooth
then  the corresponding estimates of $L^q$-norm for $\partial_t u$ and $D^2u$ are true  in the whole cylinder $Q$.

 \vspace{0.2cm}

The paper is organized as follows. In Section 2 we introduce notations used in this paper, describe the different components of the free boundary and formulate the main result of the paper: Theorem~\ref{main-thm}. In Section~3 we show the continuity of the time-derivative $\partial_t u$ across the special part of the free boundary where the spatial gradient $Du$ does not vanish, and estimate $|\partial_tu|$ on this part unformly by a constant depending only on  given quantities. Further, in Section~4 we verify that positive and negative parts of the space directional derivatives $D_eu$ for any direction $e\in \mathbb{R}^n$ are sub-caloric outside some "pathological" part of the free boundary. We use this information in Section~5 for proving the quadratic growth estimates which are crucial for the final estimates of the higher order derivatives. The uniform $L^{\infty}$-estimates of $\partial_t u$ and $D^2u$ depending on given quantities and on the distance  to the "pathological" part of the free boundary are obtained in Section~6. Finally, in Section~7 we state and prove some preliminary facts which are used intensively for proving of almost all results in the previous sections.



\section{Notation and Preliminaries.}
Throughout this article we use the following notation:

\noindent $z=(x,t)$ are points in ${\mathbb R}^{n+1}_{x,t}$, where $x\in \mathbb{R}^n$, $n \geqslant 1$, and $t\in \mathbb{R}^1$;
%

\noindent $x=(x_1, x')=(x_1,x_2,\dots, x_n)$,  if $n\geqslant 2$;

\noindent $|x|$ is the Euclidean norm of $x$;



\noindent $B_r(x^0)$ denotes the open ball in ${\mathbb R}^n$ with
center $x^0$ and radius $r$; 

\noindent $Q_r(z^0)=Q_r(x^0,t^0)=B_r(x^0) \times ]t^0-r^2,t^0+r^2[$;

\noindent $Q_r^-(z^0)=Q_r(z^0) \cap \left\lbrace t <t^0\right\rbrace$. 


\noindent When omitted, $x^0$ (or $z^0=(x^0,t^0)$, respectively) is assumed to be the origin.

\noindent $\partial 'Q_r(z^0)$ or $\partial ' Q_r^{-}(z^0)$ denote the parabolic boundary of the corresponding cylinder, i.e., the topological boundary minus the top of the cylinder.


For a cylinder $Q=\mathcal{U} \times ]0,T[$ and any $\epsilon>0$ we define the corresponding cylinder $Q^{\epsilon}$ as
$$
Q^{\epsilon}=\mathcal{U}^{\epsilon}\times ]\epsilon^2, T[,
$$
where $\mathcal{U}^{\epsilon} \subset \mathcal{U}$ and $\textit{dist}\,\left\lbrace \mathcal{U}^{\epsilon}, \partial\mathcal{U}\right\rbrace \geqslant \varepsilon$.







\noindent $u_+=\max\, \left\lbrace u,0\right\rbrace $; \qquad $u_-=\max\, \left\lbrace -u,0\right\rbrace $;

\noindent $D_i$ denotes the differential operator with respect to $x_i$;

\noindent $D=(D_1, D')=(D_1,D_2,\dots , D_n)$ denotes the spatial
gradient;

\noindent $D^2u=D(Du)$ denotes the Hessian of $u$;

\noindent $\partial_tu=\dfrac{\partial u}{\partial t}$.
\vspace{0.1cm}

\noindent $D_{\nu}$ stands for the operator of differentiation along a direction $\nu \in \mathbb{R}^n$, i.e., $|\nu|=1$ and
$
D_{\nu}u=\sum\limits_{i=1}^n \nu_iD_iu.
$

We adopt the convention that the indices $i,j,l$ always vary from $1$ to $n$.
We also adopt the convention regarding summation with respect to repeated indices.


\vspace{0.2cm} \noindent $\|\cdot\|_{p,\,\mathcal{D}}$ denotes the norm in
$L^p(\mathcal{D})$, $1<p \leqslant \infty$;

\vspace{0.2cm}
\noindent $W^{2,1}_p(\mathcal{D})$ and $W^{1,0}_p(\mathcal{D})$ are  anisotropic Sobolev spaces with the
norms
\begin{gather*}
\|u\|_{W^{2,1}_p(\mathcal{D})}=\|\partial_t u \|_{p,\,\mathcal{D}}+
\|D^2u\|_{p,\,\mathcal{D}}+\|u\|_{p,\,\mathcal{D}},\\
\|u\|_{W^{1,0}_p(\mathcal{D})}=\|Du\|_{p,\,\mathcal{D}}+\|u\|_{p,\,\mathcal{D}},
\end{gather*}
respectively. \vspace{0.2cm}

\noindent For a cylinder $\mathcal{Q}=\mathcal{U} \times ]T_1,T_2[ \subset \mathbb{R}^n_x \times\mathbb{R}^1_t$ we denote by $V_{2}(\mathcal{Q})$  the Banach space consisting of all elements of $W^{1,0}_2(\mathcal{Q})$ with a finite norm
$$
\|u\|_{V_2(\mathcal{Q})}=\sup\limits_{T_1<t< T_2}\|u\|_{2,\,\mathcal{U} }+\|Du\|_{2,\, \mathcal{Q}}.
$$


\noindent
$\fint\limits_{\mathcal{D}}\dots$ stands for the average integral over the set $\mathcal{D}$, i.e.,
$$
\fint\limits_{\mathcal{D}} \dots =\frac{1}{\text{meas}\,\left\lbrace \mathcal{D}\right\rbrace }\int\limits_{\mathcal{D}} \dots .
$$



We say that $\xi=\xi(x,t)$ is a cut-off function for a cylinder $Q_r(\hat{z})$ if 
$$
\xi(x,t)=\xi_1(x) \xi_2 (t),
$$
where $\xi_i \geqslant 0$, $i=1,2$, 
$$
\xi_1 \in C^{\infty}_{0} \left( B_r(\hat{x})\right) , \qquad \xi_1\equiv 1 \quad \text{in}\quad B_{r/2}(\hat{x}),
$$
 while $\xi_2 \in C^{\infty}([\hat{t}-r^2, \hat{t}])$, 
$\xi_2(\hat{t}-r^2)=0$ and
$\xi_2(t) \equiv 1$ for $t \geqslant \hat{t}-r^2/4$.
\vspace{0.2cm}

We define the parabolic distance $\textit{dist}_p$ from a point $z=(x,t)$ to a set $\mathcal{D} \subset\mathbb{R}^{n+1}$ by
$$
\textit{dist}_p \left( z, \mathcal{D}\right) :=\sup \left\lbrace r>0 : Q^-_r(z)\cap \mathcal{D}=\emptyset \right\rbrace .
$$






We use letters $M$, $N$, $C$ and $c$ (with or without sub-indices) to
denote various constants. To indicate that, say, $C$ depends on some
parameters, we list them in the parentheses: $C(\dots)$. We do not indicate the dependence of constants on $n$. In addition, we will write \textit{sup} instead of \textit{ess\,sup} and \textit{inf} instead of \textit{ess\,inf}.

\vspace{0.2cm}
We denote 
\begin{align*}
\Omega_{\pm}(u)&:=\left\lbrace z\in Q, \ \text{where} \ h[u](z)=\pm 1\right\rbrace ,\\
\Gamma (u)&:=\partial\Omega_+ \cap \partial\Omega_- \ \text{is the free boundary}.
\end{align*}
The latter means that $\Gamma(u)$ is the set where the function $h[u](z)$ has a jump.

\vspace{0.2cm}
We also introduce special notation for the different parts of $\Gamma (u)$
\begin{align*} 
\Gamma_{\alpha}(u)&:=\Gamma (u) \cap \left\lbrace u=\alpha\right\rbrace ,\\
\Gamma_{\beta} (u)&:=\Gamma (u) \cap \left\lbrace u=\beta\right\rbrace .
\end{align*}
By definition,
$$
\left\lbrace u \leqslant \alpha \right\rbrace \subset \Omega_{-}\quad \text{and}\quad \left\lbrace u \geqslant \beta\right\rbrace  \subset \Omega_{+}.
$$
\begin{figure}[htbp]
\centering
\includegraphics[width=0.85\textwidth]{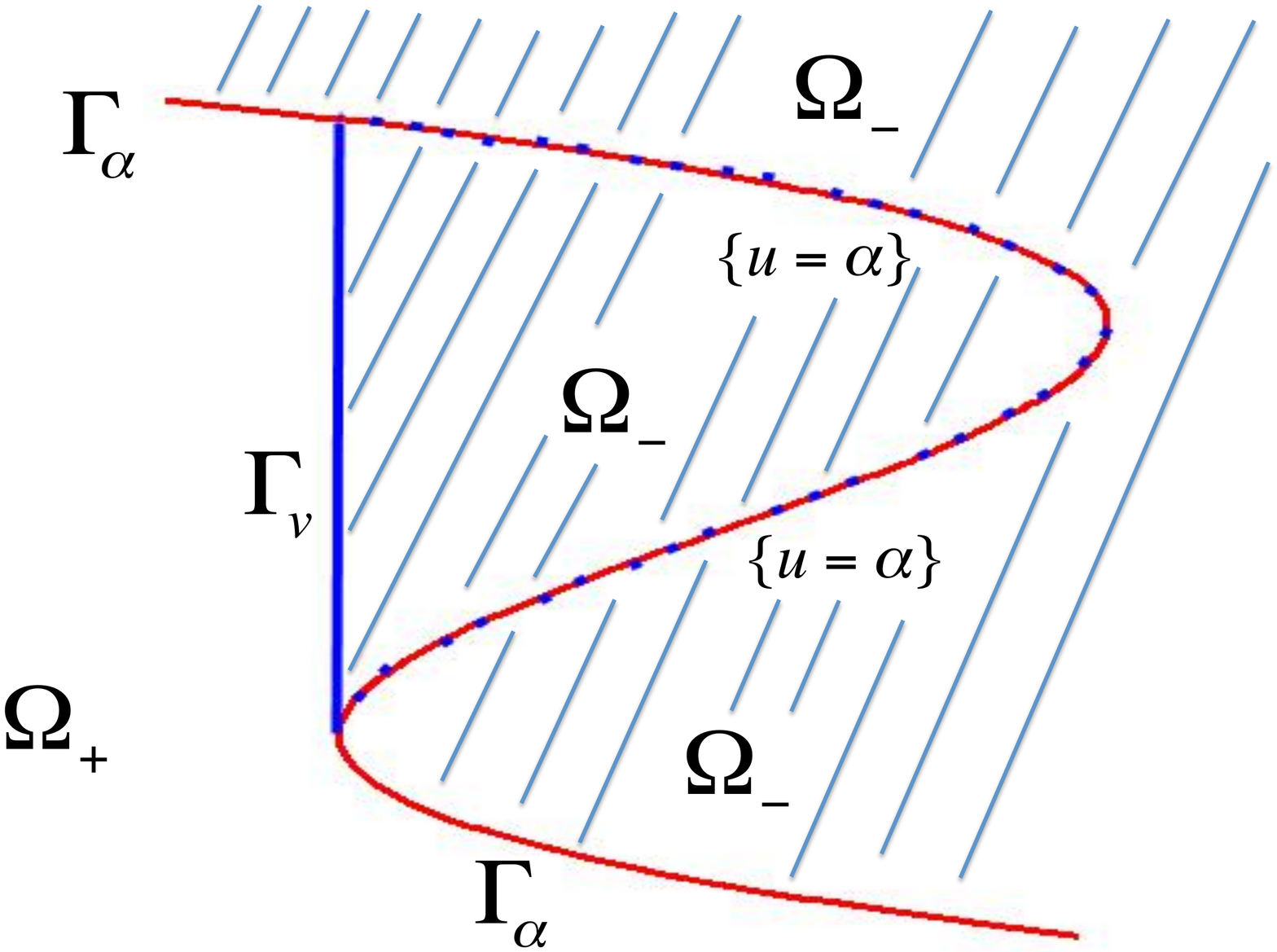}
\caption{Structure of the free boundary for $n=1$}
\label{FB-structure}
\end{figure}

It is also easy to see that the sets $\left\lbrace u=\alpha\right\rbrace $ and $\left\lbrace u=\beta\right\rbrace $ are separated from each other. 

\begin{remark} \label{choose-of-rho}
In any cylinder $Q^{\epsilon}$ the distance from the level set $\left\lbrace u=\alpha\right\rbrace $ to the level set $\left\lbrace u=\beta\right\rbrace $ is estimated from below by a positive constant depending on $M$, $\epsilon$ and $\beta -\alpha$ only.
\end{remark}

Observe that the level sets $\left\lbrace u=\alpha \right\rbrace $ and $\left\lbrace u=\beta\right\rbrace $ are not alsways the parts of the free boundary $\Gamma (u)$. Indeed, if the level set $\left\lbrace u=\alpha \right\rbrace $ is locally not a $t$-graph, then a part of $\left\lbrace u=\alpha\right\rbrace $ may occur inside $\Omega_{-}$. In this case $\Gamma (u)$ 
may contain several components of $\Gamma_{\alpha}$ connected by cylindrical surfaces with generatrixes parallel to $t$-axis (see Figure~\ref{FB-structure}). Similar statement is true for the level set $\left\lbrace u=\beta\right\rbrace$.  We will denote by $\Gamma_{v}$ the set of all points $z$ lying in such  vertical parts of $\Gamma (u)$. It should be noted that $\Gamma_v$ is, in general, not the level set $\left\lbrace u=\alpha\right\rbrace $ as well as not the level set $\left\lbrace u=\beta\right\rbrace $. This $\Gamma_v$ is just the  "pathological" part of the free boundary that we have mentioned in Introduction. Thus, we have
$$
\Gamma (u)=\Gamma_{\alpha} (u) \cup \Gamma_{\beta} (u) \cup \Gamma_{v}.
$$
\vspace{0.2cm}


We will also distinguish the following parts of $\Gamma$:
$$
\Gamma_{\alpha}^0 (u)=\Gamma_{\alpha} (u) \cap \left\lbrace |Du|=0\right\rbrace , \qquad \Gamma_{\alpha}^* (u)=\Gamma_{\alpha} (u) \setminus \Gamma_{\alpha}^0 (u).
$$
The sets $\Gamma_{\beta}^0$ and $\Gamma^*_{\beta}$ are defined analogously. In addition, we set
$$
\Gamma^0(u):=\Gamma^0_{\alpha}(u) \cup \Gamma^0_{\beta} (u), \quad
\Gamma^*(u):=\Gamma^*_{\alpha}(u) \cup \Gamma^*_{\beta} (u).
$$

\begin{remark}
It is obvious that $u\in C^{\infty}$ in the interior of the sets $\Omega_{\pm}$.
\end{remark}

\vspace{0.2cm}

Now we formulate the main result of the paper.
\begin{thm}\label{main-thm}
Let $u$ be a (strong) solution of Eq. (\ref{main-equation}), and let $z\in Q\setminus \Gamma (u)$. Then
$$
|\partial_tu(z)|+|D^2u(z)|\leqslant C(\rho_0, \varepsilon, M, \beta-\alpha ).
$$
Here $\rho_0:=\textit{dist}_p
\left\lbrace z, \Gamma_v\right\rbrace $ and $\epsilon:=\textit{dist}_p \left\lbrace z, \partial 'Q\right\rbrace $.
\end{thm}

\begin{proof}
The proof of this statement follows from Lemmas~\ref{estimate-u_t-beyond-Gv} and \ref{estimate-for-D^2u}.
\end{proof}

\section{Estimates of $\partial_t u$ on $\Gamma^*(u)$}

\begin{lemma} \label{one-sided-estimates-u_t}
Let $u$ be a  solution of Eq. (\ref{main-equation}), and let $Q^-_{3\rho}(z^*)$ be an arbitrary cylinder contained in $Q$.  Then we have the estimates
\begin{align}
\inf\limits_{Q_{\rho}^-(z^*)}\partial_t u &\geqslant -N, \quad  \ \text{provided that} \quad  Q_{3\rho}^-(z^*) \cap\Gamma_{\beta}=\emptyset, \label{inf-u_t} \\
\sup\limits_{Q_{\rho}^-(z^*)}\partial_t u &\leqslant N, \qquad  \text{provided that}\quad Q_{3\rho}^-(z^*) \cap \Gamma_{\alpha}=\emptyset. \label{sup-u_t}
\end{align}
Here $N=N(M, \rho)$.
\end{lemma}


\begin{proof} 

Assume for the definiteness that $z^*$ lyies in a neighborhood of $\Gamma_{\beta}$. 
Consider in $Q_{2\rho }^{-}(z^*)$ the difference quotient of $u$ in the $t$-direction, i.e.,
$$
u^{(\tau)}(x,t)=\frac{u(x,t)-u(x,t-\tau)}{\tau}
$$
with some small positive $\tau$. To prove (\ref{sup-u_t}) it is sufficient to get the corresponding estimate for $u^{(\tau)}$ uniformly with respect to $\tau$. \vspace{0.2cm}


Further, we observe that equation (\ref{main-equation}) and integration by parts provide for all test-finctions $\eta \in W_2^{1,0}(Q_{2\rho}^-(z^*))$ vanishing on $\partial B_{2\rho}(x^*) \times [t^*-4\rho^2,t^*]$ the validity of the following integral identity
\begin{equation}\label{first-identity}
\int\limits_{Q_{2\rho}^-(z^*)} \left( \partial_t u \eta +DuD\eta \right) dxdt=-\int\limits_{Q_{2\rho}^-(z^*)}h[u]\eta dxdt.
\end{equation}

Using the same reasonings as in deriving of (\ref{first-identity}) we get for all test-functions 
$\widetilde{\eta}\in W_2^{1,0}\left( Q_{2\rho}^-\left( x^*, t^*+\tau\right)\right)  $ that are equal to zero on $\partial' Q_{2\rho}^-\left( x^*,t^*+\tau\right) $ the integral identity
\begin{equation} \label{shifted-identity}
\int\limits_{Q_{2\rho}^-(x^*,t^*+\tau)} \left( \partial_t u \widetilde{\eta} +Du D\widetilde{\eta}\right) dxdt=-\int\limits_{Q_{2\rho}^-(x^*,t^*+\tau)}h[u]\widetilde{\eta} dxdt.
\end{equation}
Putting in (\ref{shifted-identity}) $\widetilde{\eta}(x,t)=\eta(x, t+\tau)$ we obtain after elementary change of variables the relation
\begin{equation} \label{second-identity}
\begin{aligned}
\int\limits_{Q_{2\rho}^-(z^*)}
\left[  \partial_t u(x, t-\tau) \right. & \left.\eta(x,t)+Du(x,t-\tau)D\eta (x,t) \right]  dxdt\\
&=-\int\limits_{Q_{2\rho}^-(z^*)}h[u](x,t-\tau)\eta(x,t)dxdt.
\end{aligned}
\end{equation}

Now, 
we substract  (\ref{second-identity}) from  (\ref{first-identity}), divide the result by $\tau$ and integrate by parts. After these transformations we arrive at the equality
\begin{equation} \label{third-identity}
\begin{aligned}
\int\limits_{Q_{2\rho}^-(z^*)}\left[ \partial_t u^{(\tau)} \eta \right. &+\left. Du^{(\tau)} D\eta\right] dxdt\\
&=-\frac{1}{\tau}\int\limits_{Q_{2\rho}^-(z^*)}
\left(  h[u](x,t)-h[u](x,t-\tau)\right)  \eta dxdt.
\end{aligned}
\end{equation}

Setting in (\ref{third-identity}) 
$$
\eta (x,t)=\left( u^{(\tau)}-k\right)_+\xi^2(x,t), 
$$
where $\xi$ is a standard cut-off function for a cylinder $Q_{2\rho}^-(z^*)$ (see Notation),
we can rewrite (\ref{third-identity}) in the form
\begin{equation} \label{4-identity}
\begin{aligned}
&\int\limits_{Q_{2\rho}^-(z^*)}\left\lbrace \partial_t u^{(\tau)} \left( u^{(\tau)}-k\right)_+\xi^2  +Du^{(\tau)} D\left[ \left( u^{(\tau)}-k\right)_+\xi^2\right]\right\rbrace dxdt \\
&=-\frac{1}{\tau}\int\limits_{Q_{2\rho}^-(z^*)} \left(  h[u](x,t)-h[u](x,t-\tau)\right)   \left( u^{(\tau)}-k\right)_+\xi^2dxdt.
\end{aligned}
\end{equation}

We claim that $h[u](x,t)-h[u](x,t-\tau)\geqslant 0$ in $Q_{2\rho}^-(z^*)$. Indeed,  we have the relation
$$
Q_{2\rho}^-(z^*) \cap \Gamma_{\alpha}=\emptyset.
$$

Recall that by definition $h[u](x,t)$ may decrease in $t$ only in a neighborhood of $\Gamma_{\alpha}$. Therefore, in $Q_{2\rho}^-(z^*)$ the function $h[u]$ is either constant or increasing one.  
The latter means that for  we have instead of (\ref{4-identity}) the inequality
\begin{equation} \label{13a}
\int\limits_{Q_{2\rho}^-(z^*)}\left\lbrace  \partial_t u^{(\tau)} \left( u^{(\tau)}-k\right)_+\xi^2  +Du^{(\tau)} D\left[ \left( u^{(\tau)}-k\right)_+\xi^2\right] \right\rbrace dxdt \leqslant 0.
\end{equation}

Observe that we may take in (\ref{13a}) the cut-off fucntion  $\xi$ multiplied by the characteristic function of an interval $[t^*-4\rho^2, t]$ with an arbitrary $t\in ]t^*-4\rho^2,t^*]$ instead of $\xi$.
This leads to the inequalities
$$
\begin{gathered}
\int\limits_{t^*-4\rho^2}^{t}\int\limits_{B_{2\rho}(x^*)}\left\lbrace  \partial_t u^{(\tau)}\left( u^{(\tau)}-k\right)_+\xi^2  + Du^{(\tau)} D\left[ \left( u^{(\tau)}-k\right)_+\xi^2\right]\right\rbrace dxdt \leqslant 0, \\
\forall t\in ]t^*-4\rho^2, t^*].
\end{gathered}
$$

Further arguments are rather standard. We leave the trivially nonnegative terms in the left-hand side of the above inequalities, while the rest terms are transferred to the right-hand side and estimated from above with the help of Young's inequality. As a consequence,  we get
\begin{equation} \label{inequality-1}
\begin{aligned}
\sup\limits_{t^*-4\rho^2<t\leqslant t^*}\int\limits_{B_{2\rho}(x^*)}& (u^{(\tau)}-k)^2_+\xi^2dx \bigg|^{t}+
\int\limits_{Q_{2\rho}^-(z^*)}\left[ D\left( (u^{(\tau)}-k)_+\right)\right] ^2\xi^2 dxdt \\
&\leqslant \int\limits_{Q_{2\rho}^-(z^*)}\left( u^{(\tau)}-k\right)_+^2\left[ 4|D\xi|^2+2\xi|\partial_t \xi|\right] dxdt.
\end{aligned}
\end{equation}
With inequalities (\ref{inequality-1}) for an arbitrary $k \geqslant 0$ at hands we may apply succesively Fact~\ref{fact-4}  with $v~=~u^{(\tau)}$ and
inequalities~(\ref{W^2_q-estimates}) with $q=2$
 which immediately imply the desired estimate~(\ref{sup-u_t}). \vspace{0.2cm}

It remains only to observe that the case of  $z^*$ lying near $\Gamma_{\alpha}$ is treated almost similarly. The only differences are that we should choose in (\ref{third-identity}) 
$$\eta (x,t)=\left( u^{(\tau)}-k\right)_- \xi^2(x,t), \quad k \leqslant 0,
$$
and then check the validity of the inequality $h[u](x,t)-h[u](x,t-\tau) \leqslant 0$ in the cylinder $Q_{2\rho}^-(z^*)$.
\end{proof}
\vspace{0.2cm}

\begin{lemma} \label{lemma-3.2}
Let $u$ be a  solution of Eq. (\ref{main-equation}) and let $z^* \in \Gamma^* \setminus \Gamma_v$. 

Then $\Gamma^* \setminus \Gamma_v$ is locally a $C^1$-surface and $\partial_t u$ is a continuous function in 
a neigborhood of $z^*$. 
\end{lemma}

\begin{proof} Continuity of $\partial_t u$ across $\Gamma^*$ can be proved by using the same arguments as in (the proof of) Lemma~7.1 \cite{SUW09}. For the readers convenience we sketch the details.

Suppose for the definiteness that $z^*\in \Gamma^*_{\alpha} \setminus \Gamma_v$. 
Without restriction it may be assumed that $D_1u(z^*)>0$. Then, in a sufficiently small cylinder $Q_{\rho}(z^*)$ satisfying $Q_{\rho}(z^*) \cap \Gamma_v$ the function $u$ is strictly increasing in $x_1$-direction.

Further, using the von Mises transformation, we introduce the new variables
$$
(x_1,x',t)\rightarrow (y,x',t),
$$
where $y:=u(x,t)-\alpha$. We also introduce the function $v$ such that 
$$x_1=v(y,x',t).
$$
Transforming in $Q_{\rho}(z^*)$ Eq. (\ref{main-equation}) for $u$ into terms of $v$ we obtain the uniformly parabolic equation
$$
\partial_t v-a^{ij}\left(\partial v\right) \partial_i(\partial_jv)=g(y)\partial_1v,
$$
where $\partial_1v:=\dfrac{\partial v}{\partial y}=\dfrac{1}{D_1u}>0$, $\partial_m v:=\dfrac{\partial v}{\partial x_m}=D_mv=-\dfrac{D_mu}{D_1 u}$, 
\begin{equation} \label{gradient-v}
\partial v=(\partial_1v, \partial'v)=\left( \frac{1}{D_1u}, -\frac{D'u}{D_1 u}\right) ,  \qquad \partial_tv :=\frac{\partial v}{\partial t}=-\frac{\partial_t u}{D_1u},
\end{equation}
$$g(y)=\left\lbrace \begin{aligned}
& 1, \quad \text{if}\ y>0\\
 -&1, \quad \text{if}\ y<0
\end{aligned} \right. ,$$
and  the coefficients $a^{ij}$ are defined as follows
\begin{equation} \label{formula-for-a_ij}
\begin{gathered}
a^{11}(p)=\frac{1+|p'|^2}{p_1^2}, \quad a^{mm}(p)=1, \quad a^{1m}(p)=a^{m1}(p)=-\frac{p_m}{p_1},\\
a^{m\widetilde{m}}(p)=0 \quad \text{if}\quad m\neq \widetilde{m}
\end{gathered}
\end{equation}
(here the indices $m$ and $\widetilde{m}$ vary from $2$ to $n$, and $p\in \mathbb{R}^n$).
\vspace{0.2cm}

Elementary calculation shows that for the difference quotient in the $t$-direction
$$
v^{(\tau)}(y,x',t):=\frac{v(y,x',t)-v(y,x',t-\tau )}{\tau}
$$ 
we have
\begin{equation} \label{eq-for-v-tau}
\partial_t v^{(\tau)}-a^{ij}\left(\partial v\right) \partial_i(\partial_jv^{(\tau)})- b^k \partial_k v^{(\tau)}=g(y)\partial_1v^{(\tau)},
\end{equation}
where $b^k:=\dfrac{\partial a^{ij}(Z_{\tau})}{\partial p_k}\partial_i
\left( \partial_j v(y,x',t-\tau)\right)$, 
$$Z_{\tau}=\vartheta (y,x',t) \partial v(y,x',t-\tau)-\left[ 1-\vartheta (y,x',t)\right] 
\partial v(y,x',t)$$ and $\vartheta (y,x',t)\in [0,1]$.
\vspace{0.1cm}

Observe that for the second derivatives of $v$ we have the relations
\begin{equation} \label{second-derivatives-v}
\begin{gathered}
\partial_1\left( \partial_1 v\right) =-\frac{D_{11}u}{(D_1 u)^3}, \quad \partial_1\left( \partial_m v\right) =\frac{D_{11}uD_mu}{|D_1u|^2}-\frac{D_{1m}u}{D_1u},\\
\begin{aligned}
\partial_m(\partial_{\tilde{m}}v)&=\frac{D_{11}uD_muD_{\tilde{m}}u}{|D_1u|^2}\left( \frac{1}{D_1u}-2\right)\\
&+\frac{D_{1m}uD_{\tilde{m}}u}{D_1u}+\frac{D_{1\tilde{m}}uD_{m}u}{D_1u}-\frac{D_{m\tilde{m}}u}{D_1u}.
\end{aligned} 
\end{gathered}
\end{equation}
According to estimates (\ref{W^2_q-estimates}) and formulas (\ref{gradient-v})-(\ref{formula-for-a_ij}) and (\ref{second-derivatives-v}) we may conclude that in Eq.~(\ref{eq-for-v-tau}) the coefficients $a^{ij}$ are H{\"o}lder continuous functions satisfying the ellipticity condition, whereas the coefficients $b^k $  are elements of $L^q$ with an arbitrary $q<\infty$. Therefore, the parabolic theory implies that $v^{(\tau)}\in C^{\sigma}$ for some $\sigma \in (0,1)$. 
We note also that all the estimates of corresponding norms are uniformly bounded in $\tau$. Hence
we immediately conclude that $\partial_tu$ is also H{\"o}lder continuous with some exponent $\sigma'$ satisfying $0<\sigma'<\sigma$.
It is also evident that near $z^*$ the free boundary $\Gamma_{\alpha}$ is a $C^1$-surface . \smallskip


It remains only to observe that  in the case $z^*\in \Gamma_{\beta}^* \setminus \Gamma_v$ we should choose the new variable $y$ in von Mises transformation as $y:=u(x,t)-\beta$ and repeat the above steps.
\end{proof}
\vspace{0.2cm}




\begin{cor} \label{corollary-Gamma^*}
Let $u$ satisfy Eq. (\ref{main-equation}).
Then for any cylinder $Q^{\epsilon}\subset  Q$ 
we have
\begin{equation} \label{u_t-on-Gamma*}
\sup \limits_{\left( \Gamma^*\setminus \Gamma_v\right)  \cap Q^{\epsilon}}|\partial_tu| \leqslant N_*(M,\epsilon, \beta -\alpha).
\end{equation}

In addition, the mixed second derivatives $D_i\left( \partial_t u\right) $ are $L^2_{loc}$-functions in $Q\setminus\left( \Gamma^0 \cup \Gamma_v\right) $.
\end{cor}

\begin{proof}
Consider for the definiteness the case $z^*\in \left(  \Gamma_{\alpha}^* \setminus \Gamma_v \right)   \cap Q^{\epsilon}$. 
Due to Lemma~\ref{lemma-3.2} a function $\partial_tu$ is continuous in a neighborhood of $z^*$. 



Recall that by definition of $\Gamma_{\alpha}$ the function $h[u]$ has a jump in $t$-direction from $+1$ to $-1$ there. The latter means that if we cross the free boundary $\Gamma_{\alpha}^*$ in positive $t$-direction then the corresponding phases change from $\Omega_{+}$ to $\Omega_{-}$. Since $u(z^*)=\alpha$ and $u(x^*,t^*-\varepsilon)> \alpha$ for any $\varepsilon>0$ we conclude that $\partial_t u(z^*) \leqslant 0$.
Hence   the inequality
\begin{equation} \label{other-side-u_t}
\partial_t u \leqslant 0 \qquad \text{on}\quad \Gamma^*_{\alpha}\setminus \Gamma_v
\end{equation}
is valid.


Now , taking into account Remark~\ref{choose-of-rho}, one may combine\  (\ref{other-side-u_t})\  with\  one-sided\  inequality\  (\ref{inf-u_t}). It gives the desired estimate~(\ref{u_t-on-Gamma*}) with $\Gamma^*_{\alpha}$ instead of the whole $\Gamma^*$.

The other case, i.e., $z^* \in \Gamma_{\beta}^* \setminus \Gamma_v$ is treated in a similar manner. It is necessary only to observe that
if we cross the free boundary $\Gamma_{\beta}^*$ in positive $t$-direction then the phases will change from $\Omega_-$ to $\Omega_+$ and, consequently, $\partial_t u(z^*) \geqslant 0$ and the inequality
\begin{equation} \label{other-side-u_t-on-Gamma-beta}
\partial_t u \geqslant 0 \qquad \text{on}\quad \Gamma^*_{\beta}\setminus \Gamma_{v}
\end{equation}
holds true. In view of Remark~\ref{choose-of-rho}, the combination of (\ref{other-side-u_t-on-Gamma-beta}) with  one-sided estimate (\ref{sup-u_t}) finishes the proof of (\ref{u_t-on-Gamma*}).

Finally, using the same arguments as in the proof of Lemma~\ref{one-sided-estimates-u_t} we may get inequality (\ref{13a}) with sufficiently small $\rho$ and any $k \geqslant -1$ which permits us to conclude that  the mixed derivatives $D_i(\partial_t u)$ belong locally to a class of $L^2$-functions.
\end{proof}

\section{Sub-Caloricity of $D_eu$}

\begin{lemma} \label{help-sub-caloricity}
Let $w\in C(\mathcal{D})\cap W^{1,0}_{2, loc}(\mathcal{D})$ with $\mathcal{D}$ being a domain in $\mathbb{R}^{n+1}$, and let the inequality
\begin{equation} \label{inequality-for-subcaloricity}
\int\limits_{\mathcal{D}}\left( -w\partial_t\eta+Dw D\eta\right) dz \leqslant 0
\end{equation}
hold for any nonnegative function $\eta \in C^{\infty}_0(\mathcal{D})$ with $\textit{supp}\, \eta \subset \left\lbrace w>0\right\rbrace $.

Then the function $w_{+}$ is sub-caloric in $\mathcal{D}$.
\end{lemma}

\begin{proof}
First, we take in  (\ref{inequality-for-subcaloricity})  nonnegative functions $\eta \in C^{\infty}_0(\mathcal{D})$ with 
\begin{equation} \label{w>delta}
\textit{supp}\, \eta \subset \left\lbrace w\geqslant \frac{\delta}{2} >0\right\rbrace.
\end{equation}
Without loss of generality we may consider instead of $w$ in (\ref{inequality-for-subcaloricity}) its mollifier $w_{\rho}$ with sufficiently small parameter $\rho$.  After integration by parts we arrive at
\begin{equation} \label{25}
\int\limits_{\mathcal{D}}\left[  \partial_t w_{\rho}\eta+Dw_{\rho} D\eta  \right] dz \leqslant 0.
\end{equation}
We set in (\ref{25}) $\eta=\psi_{\delta}(w_{\rho}) \varphi$, where $\varphi \in C^{\infty}_0 (\mathcal{D})$ is an arbitrary nonnegative test function, while 
$$
\psi_{\delta}(s)=\left\lbrace \begin{array}{cl}
0, & \text{if} \ s \leqslant\delta \\
\dfrac{(s-\delta)}{\delta}, & \text{if} \ \delta <s <2\delta \\
1, & \text{if} \ s \geqslant 2\delta
\end{array}.
\right.
$$
Observe that such a choice of $\eta$ is not restrictive, since due to definition of $\psi_{\delta}$ we have for sufficiently small  $\rho$ the evident inclusions
$$
\textit{supp}\, \eta \subset \left\lbrace w_{\rho} \geqslant \delta\right\rbrace  \subset \left\lbrace w>\frac{\delta}{2}\right\rbrace .
$$

After substitution of $\eta$ inequality (\ref{25}) takes the form
\begin{equation} \label{26}
\int\limits_{\mathcal{D}}\left[ \partial_t w_{\rho} \psi_{\delta}(w_{\rho})\varphi
+|Dw_{\rho}|^2\psi_{\delta}'(w_{\rho})\varphi+Dw_{\rho} \psi_{\delta}(w_{\rho})D\varphi\right] dz \leqslant 0. 
\end{equation}
Elementary calculation shows that $\partial_t w_{\rho}  \psi_{\delta}(w_{\rho})=\frac{d}{dt}F_{\delta}(w_{\rho})$ where the function 
$F_{\delta}$ is defined as
$$
F_{\delta}(s)=\int\limits_0^s \psi_{\delta}(\tau)d\tau=\left\lbrace 
\begin{array}{cl}
0, & \text{if} \ s \leqslant\delta \\
\dfrac{(s-\delta)^2}{2\delta}, & \text{if} \ \delta <s <2\delta \\
s-(3/2)\delta, & \text{if} \ s \geqslant 2\delta
\end{array}.
\right.
$$
So, again integrating by parts and taking into account that the second term in (\ref{26}) is nonnegative we get the inequality
\begin{equation} \label{27}
\int\limits_{\mathcal{D}}\left[  -F_{\delta}(w_{\rho})\partial_t\varphi+ Dw_{\rho}\psi_{\delta}(w_{\rho}) D\varphi \right] dz \leqslant 0.
\end{equation}

Tending in (\ref{27}) $\rho\rightarrow 0$ and taking into account the definitions of $\psi_{\delta}$ and $F_{\delta}$ we arrive at
$$
\int\limits_{\left\lbrace w>2\delta\right\rbrace }\left[ -w\partial_t \varphi+Dw D\varphi \right] dz \leqslant 
\int\limits_{\left\lbrace \delta <w <2\delta\right\rbrace } |Dw D\varphi |dz +C\delta.
$$
Letting $\delta \rightarrow 0$ in the above inequality provides the inequality
\begin{equation} \label{28}
\int\limits_{\left\lbrace w>0\right\rbrace }\left[ -w\partial_t\varphi +Dw D \varphi \right] dz \leqslant 0.
\end{equation}
It remains only to recall that $\varphi$ in (\ref{28}) is an arbitrary nonnegative test-function. This completes the proof.
\end{proof}

\vspace{0.2cm}

\begin{lemma} \label{sub-caloricity}
Let $u$ be a solution of Eq. (\ref{main-equation}).
Then for any direction $e\in \mathbb{R}^n$ functions $\left( D_eu\right)_{\pm}$ are sub-caloric in $Q \setminus \Gamma_v$. 
\end{lemma}

\begin{proof}
Due to Lemma~\ref{help-sub-caloricity} it sufficies to check that for $w=D_eu$ inequality (\ref{inequality-for-subcaloricity}) holds true for any nonnegative function $\eta \in C_0^{\infty}(Q\setminus \Gamma_v)$ with $\textit{supp}\,\eta \subset \left\lbrace D_e u>0\right\rbrace $.

It follows from Eq. (\ref{main-equation})  that functions $D_eu$  satisfy in $Q$ the equation
\begin{equation} \label{equation-for-D_eu}
H\left[ D_eu\right] =D_e\left( h[u]\right) 
\end{equation}
in the weak (distributional) sence. Hence we obtain 
\begin{align*}
\int\limits_{Q } D_eu\left( \partial_t \eta+\Delta \eta\right) dz&=-\int\limits_{Q}h[u]D_e\eta dz
=-\int\limits_{\Omega_+}D_e\eta dz+\int\limits_{\Omega_-}D_e\eta dz\\
&=2\int\limits_{\Gamma^*} \eta \cos{\left( \widehat{\mathbf{n},\mathbf{e}}\right) }d\mathcal{H}^{n},
\end{align*}
where $\mathbf{n}=\mathbf{n}(z)$ is  the unit normal vector to $\Gamma^*$ directed into $\Omega_+$, $\mathbf{e}:=(e,0)$, and $\mathcal{H}^{n}$ stands for the $n$-dimensional Hausdorff measure.

It is easy to see that the normal vector $\mathbf{n}$ has on $\Gamma^*$ the following representation
\begin{equation} \label{formula-for-normal}
\mathbf{n}(z)=\left( \frac{Du(z)}{\sqrt{|Du(z)|^2+(\partial_t u(z))^2}},\frac{\partial_t u(z)}{\sqrt{|Du(z)|^2+(\partial_t u(z))^2}}\right) .
\end{equation}
Indeed, since $u>\alpha$ in $\Omega_+$ and $\Gamma_{\alpha} \subset \left\lbrace u=\alpha\right\rbrace $, the vector $Du(z)$ at $z\in \Gamma^*_{\alpha}$ is directed into $\Omega_+$. In addition, 
we recall (see (\ref{other-side-u_t})) that $\partial_t u \leqslant 0$ on $\Gamma^*_{\alpha}$. Therefore, the projection of $\mathbf{n}$ from formula (\ref{formula-for-normal}) on the $t$-axis is also nonpositive. Because of $\Omega_+$ is locally a subgraph of $\Gamma_{\alpha}$ in $t$-direction, we conclude that on $\Gamma^*_{\alpha}$ the whole  vector $\mathbf{n}$ defined by (\ref{formula-for-normal})  is directed into $\Omega_+$. 
Similarly, we have $\left\lbrace u < \beta\right\rbrace$ in $\Omega_-$ and $\Gamma_\beta \subset \left\lbrace u=\beta\right\rbrace $. Therefore, the spatial gradient $Du(z)$ at $z\in \Gamma^*_{\beta}$ is directed into $\Omega_+$. 
Moreover, 
on $\Gamma^*_{\beta}$ we have $\partial_t u \geqslant 0$ (see (\ref{other-side-u_t-on-Gamma-beta})) and $\Omega_+$ is a $t$-epigraph of $\Gamma^*_{\beta}$. So, 
the  vector $\mathbf{n}$ from formula (\ref{formula-for-normal})  is again directed  into $\Omega_+$. 

Now, taking into account the inclusion $\textit{supp}\,\eta \subset \left\lbrace D_eu >0\right\rbrace$ and representation (\ref{formula-for-normal}) we conclude that 
$$
\eta \cos{\left( \widehat{\mathbf{n(z)},\mathbf{e}}\right) } \geqslant 0 \qquad \forall z\in \Gamma^* 
$$
and complete the proof.
\end{proof}
\vspace{0.2cm}

\begin{remark}
We emphasize that $\left( D_eu\right)_{\pm}$ are, in general, not sub-caloric near~$\Gamma_v$. 
\end{remark}

\section{Quadratic Growth Estimates}

\begin{lemma}  \label{quadratic-lemma}
Let $u$ satisfy (\ref{main-equation}), let  $z^0 \in \Gamma^0 $,  and let 
$$
\textit{dist}_p\left\lbrace z^0, \Gamma_{v}\right\rbrace \geqslant \rho_0>0, \qquad \quad \textit{dist}_p\left\lbrace z^0, \partial'Q\right\rbrace \geqslant \rho_0.
$$ 
There exists a positive constant $C_0$ completely defined by the values of $\rho_0$ and $M$ such that
\begin{equation} \label{4-1}
\underset{Q_r^-(z^0)}{\text{osc}}\, u\leqslant C_0 r^2\qquad \text{for all}\quad r\leqslant \rho_0.
\end{equation}
\end{lemma}

\begin{proof}
We verify inequality (\ref{4-1}) for $z^0 \in \Gamma^0_{\alpha}$. The other case, i.e., $z^0 \in \Gamma^0_{\beta} $ can be proved by using similar arguments.

We argue by contradiction. Suppose (\ref{4-1}) fails. Then there exist a sequence $r_k >0$ as well as sequences $u_k$ of solutions to (\ref{main-equation}) satisfying (\ref{sup-estimate}), and points $z^k \in \Gamma_{\alpha}^0 (u_k)$   such that for all $k \in \mathbb{N}$ we have
$$
\textit{dist}_p \left(  z^k, \Gamma_v (u_k) \right)   \geqslant \rho_0, \qquad \quad \textit{dist}_p \left(  z^k, \partial'Q \right)   \geqslant \rho_0
$$ 
and
\begin{equation} \label{4-3}
\sup\limits_{Q^-_{r_k}(z^k)}|u_k-\alpha |\geqslant kr_k^2.
\end{equation}

Thanks to assumption (\ref{sup-estimate}) the left-hand side of (\ref{4-3}) is bounded by $2M$ and, consequently, $r_k\rightarrow 0$ as $k\rightarrow \infty$. It is evident that we can choose $r_k$ as the maximal value of $r$ for which
$$
\sup\limits_{Q^-_{r}(z^k)}|u_k-\alpha |\geqslant kr^2.
$$
In other words, we have the relations
\begin{equation} \label{4-4}
\left\lbrace \begin{aligned}
\mathcal{M}_r(z^k, u_k)&:=\sup\limits_{Q^-_{r}(z^k)}|u_k-\alpha |<kr^2 \quad \text{for all} \ r\in (r_k,\rho_0],\\
\mathcal{M}_{r_k}(z^k,u_k)&=kr_k^2.
\end{aligned} \right.
\end{equation}
Next, we define a scaling $\tilde{u}_k$ as
$$
\tilde{u}_k(x,t)=\frac{u_k(x^k+r_kx, t^k+r_k^2t)-\alpha}{\mathcal{M}_{r_k}(z^k,u_k)}
$$
for $(x,t)\in Q^-_{\rho_0/r_k}$. Then $\tilde{u}_k$ has the following properties
\begin{equation} \label{4-5}
\sup\limits_{Q_1^-}|\tilde{u}_k|=1,
\end{equation}
\begin{equation} \label{4-6}
\tilde{u}_k(0,0)=0, \qquad |D\tilde{u}_k (0,0)|=0, 
\end{equation}
\begin{equation} \label{4-7}
\|H[\tilde{u}_k]\|_{\infty, Q^-_{1/r_k}}\leqslant \frac{r_k^2}{\mathcal{M}_{r_k}(z^k,u_k)}= \frac{1}{k}\rightarrow 0 \ \text{as}\ k\rightarrow \infty.
\end{equation}
In addition, due to (\ref{4-4}) we have for $R \in (1, \rho_0/r_k]$ the inequality
\begin{equation} \label{4-8}
\sup\limits_{Q_R^-}|\tilde{u}_k|=\frac{\mathcal{M}_{r_kR}(z^k, u_k)}{\mathcal{M}_{r_k}(z^k,u_k)}<\frac{k\left( r_kR\right)^2 }{kr_k^2}=R^2.
\end{equation} \smallskip

Now, by (\ref{4-5})-(\ref{4-8}) we will have a subsequence of $\tilde{u}_k$ weakly converging in $W^{2,1}_{q, loc}\left( \mathbb{R}^{n+1}_{x,t}\cap \left\lbrace t\leqslant 0\right\rbrace \right) $, $q< \infty$, to a caloric function $u_0$ satisfying
\begin{gather*}
\sup\limits_{Q^-_R}|u_0| \leqslant R^2 \qquad \forall R \geqslant 1,\\
u_0(0,0)=|Du_0 (0,0)|=0,
\end{gather*}
\begin{equation} \label{sup-estimate-u_0}
\sup\limits_{Q_1^-}|u_0|=1. 
\end{equation}
According to the Liouville theorem (see, for example, Lemma~2.1 \cite{ASU00}), there exist constants $a^{ij}$ such that 
\begin{equation} \label{result-of Liouville}
u_0(x,t)=a^{ij}x_ix_j+2\left( \sum\limits_{i=1}^n a^{ii}\right) t\quad \text{in}\ \mathbb{R}^{n+1}_{x,t}\cap \left\lbrace t \leqslant 0\right\rbrace .
\end{equation}

On the other hand, due to inequalities (\ref{W^2_q-estimates}),  Lemma~\ref{sub-caloricity} and Fact~\ref{ACF} we may conclude that for any direction $e\in \mathbb{R}^n$ and for all $k\in \mathbb{N}$ such that $r_k \leqslant \rho_0$
\begin{equation}\label{monotonicity-a}
\Phi (r_k, \left( D_eu_k\right)_ +, \left( D_eu_k\right)_-, \xi_{\rho_0, z^k}, z^k) \leqslant c(\rho_0),
\end{equation}
where $c(\rho_0)$ is defined completely by the values of $\rho_0$ and $M$. More precisely, by $c(\rho_0)$ we may take a majorant of the right-hand side of inequality (\ref{local-monotonicity-formula}) calculated for $\theta_1=\left(D_eu_k\right)_+$ and $\theta_2=\left( D_eu_k\right)_-$.   
After simple rescaling (\ref{monotonicity-a}) takes the form
\begin{equation} \label{rescaled-ACF}
\Phi (1, \left( D_e\tilde{u}_k\right)_+, \left( D_e\tilde{u}_k\right)_-, \zeta^k,0,0)\leqslant c(\rho_0)\left( \frac{r_k^2}{\mathcal{M}_{r_k}(z^k,u_k)}\right)^4=\frac{c(\rho_0)}{k^4}, 
\end{equation}
where for brevity  we denote  the corresponding cut-off function $\xi_{\rho_0/r_k, (0,0)}$ by $\zeta^k$. Observe that 
$\zeta^k \equiv 1$ in $B_{\rho_0/(2r_k)}$. In addition, $B_{\rho_0/(2r_k)} \supset B_1$ if $k$ is big enough, while for $\varepsilon>0$ (small and fixed) we have
$$
G(x,-t) \geqslant N(n, \varepsilon)>0 \quad \text{for} \quad -1<t<-\varepsilon, \quad x\in B_{1}.
$$
Hence,
\begin{equation} \label{thesis-2.40}
N(n,\varepsilon) \int\limits_{-1}^{-\varepsilon}\int\limits_{B_{1}} |\left( D_e\tilde{u}_k\right)_{\pm}|^2dxdt \leqslant
\int\limits_{-1}^0 \int\limits_{\mathbb{R}^n} |D_e\left( (\tilde{u}_k)_{\pm}\zeta^k\right)|^2 G(x,-t)dxdt.  
\end{equation}
Next, using (\ref{thesis-2.40}) and invoking the Poincare inequality we may reduce (\ref{rescaled-ACF}) to
\begin{align*}
\int\limits_{-1}^{-\varepsilon}\int\limits_{B_{1}}|\left( D_e\tilde{u}_k\right) _+-m^k_+(t)|^2dxdt
\int\limits_{-1}^{-\varepsilon}&\int\limits_{B_{1}}|\left( D_e\tilde{u}_k\right)_--m^k_-(t)|^2dxdt\\
&\leqslant N^{-2}(n, \varepsilon) \frac{c(\rho_0)}{k^4},
\end{align*}
where $m^k_{\pm}(t)$ denotes the corresponding average of $\left( D_e\tilde{u}_k\right)_{\pm}$ on $t$-sections over $B_{1}$.

Letting $k$ tend to infinity (and then $\varepsilon$ tend to zero), we obtain
\begin{equation} \label{Thesis-2.41}
\int\limits_{Q^-_{1}} |\left( D_eu_0\right) _+-m^+|^2dxdt \int\limits_{Q^-_{1}} |\left( D_eu_0\right)_--m^-|^2dxdt=0,
\end{equation}
where $m^{\pm}$ is the corresponding average of $\left( D_eu_0\right)_{\pm}$ over $B_{1}$. Observe that, due to representation (\ref{result-of Liouville}),  $m^{\pm}$ do not depend on $t$.  

Obviously, (\ref{Thesis-2.41}) implies that $D_eu_0$ does not change its sign in $Q^-_{1}$. 
Recall that  $e$ is an arbitrary direction in $\mathbb{R}^n$ and $u_0$ is a polinomial of the form (\ref{result-of Liouville}). It means, in particulary, that
$
u_0\equiv 0$ in $Q^-_1$. The latter contradicts (\ref{sup-estimate-u_0}) and complete the proof of (\ref{4-1}).
\end{proof}
\vspace{0.2cm}

We will need  the extension of Lemma~\ref{quadratic-lemma} to the "upper half-cylinders" $Q_r (z^0)\cap [t^0, t^0+r^2]$ as well.

\begin{lemma} 
Let all the assumptions of Lemma~\ref{quadratic-lemma} be valid. Then 
\begin{equation} \label{quadratic-whole}
\underset{Q_r(z^0)}{\text{osc}}\, u \leqslant C_1 r^2 \qquad \text{for all}\quad r \leqslant \rho_0,
\end{equation}
where $\rho_0$ is the same constant as in Lemma~\ref{quadratic-lemma} and $C_1=C_1(\rho_0, M)$.
\end{lemma}

\begin{proof}
To obtain estimate (\ref{quadratic-whole}) for $\left\lbrace t>t^0\right\rbrace $ we consider the barrier function
$$
w(x,t)=C'(\rho_0,M)\left\lbrace |x-x^0|^2+2n(t-t^0)\right\rbrace + (t-t^0) ,
$$
where $C'(\rho_0, M)=\max\left\lbrace C_0, M\rho_0^{-2}\right\rbrace$ and $C_0=C_0(\rho_0,M)$ is the constant from Lemma~\ref{quadratic-lemma}. 
Using (\ref{4-1}) for $t=t^0$ and the comparison principle one can easily verify that
\begin{equation}\label{4-2}
|u(x,t)| \leqslant w(x,t) \qquad \text{in}\quad B_{\rho_0}(x^0) \times ]t^0, t^0+r^2].
\end{equation}
Combination of (\ref{4-1}) and (\ref{4-2}) finishes the proof of (\ref{quadratic-whole}).
\end{proof}
\vspace{0.2cm}

\begin{lemma} \label{linear-growth-Du}
Let all the assumptions of Lemma~\ref{quadratic-lemma} be valid. Then
\begin{equation} \label{estimate-Du}
\sup\limits_{Q_r(z^0)} |Du| \leqslant C_2r \qquad \text{for all}\quad r\leqslant \rho_0,
\end{equation}
where $\rho_0>0$ is just the same  as in Lemma~\ref{quadratic-lemma}, while $C_2$ is a positive constant completely defined by the values of $M$ and $\rho_0$.
\end{lemma}

\begin{proof} We verify (\ref{estimate-Du}) for $z^0 \in \Gamma^0_{\alpha}$. The case $z^0 \in \Gamma^0_{\beta}$ is treated in a similar manner.

Let us choose an arbitrary $r \leqslant \rho_0/2$  and consider a point $\tilde{z} \in Q_r(z^0)$. 
Further, we take  identity (\ref{first-identity}) with $Q_{2\rho}^-(z^*)$ replaced by $Q_r^-(\tilde{z}))$ and plug in this identity a test-function
$$
\eta (x,t)=(u(x,t)-\alpha) \xi^2(x)
$$
where  $\xi \in C^{\infty}_0(B_r(\tilde{x}))$ satisfying $0 \leqslant \xi\leqslant 1$ and $|D\xi|\leqslant cr^{-1}$. After standard transformations we get the  inequality
\begin{equation} \label{4-4a}
\begin{aligned}
\int\limits_{B_r(\tilde{x})} (u-\alpha)^2&\xi^2dx \bigg|^{\tilde{t}}+
\int\limits_{Q_r^-(\tilde{z})} |Du|^2 \xi^2dxdt \leqslant \int\limits_{B_r(\tilde{x})} (u-\alpha)^2\xi^2dx \bigg|^{\tilde{t}-r^2}\\  
&+
c \int\limits_{Q_r^-(\tilde{z})} (u-\alpha)^2|D\xi|^2  dxdt
+c\int\limits_{Q_r^-(\tilde{z})} |u-\alpha|\xi^2 dxdt,
\end{aligned}
\end{equation}
where $c$ stands for an absolute constant.

In view of (\ref{quadratic-whole}) the right-hand side of (\ref{4-4}) can be estimated from above by $2c\,C_1(\rho_0,M)r^{n+4}$ which guarantees
$$
\int\limits_{Q_r^-(\tilde{z})} |Du|^2\xi^2dxdt \leqslant 2c\,C_1r^{n+4}.
$$
It remains only to observe that combination of the latter inequality with Eq.~(\ref{equation-for-D_eu}) and Fact~\ref{fact-for-estimate-Du}
 implies the estimate
$$
|Du(\tilde{z})| \leqslant \tilde{c}\,r
$$
which completes the proof.
\end{proof}

\section{Estimates of $\partial_tu$ and $D^2u$ beyond~$\Gamma_v$}

In this section we obtain the estimates of $|\partial_tu (\hat{z})|$ and $|D^2u(\hat{z})|$ in any $\hat{z}$ being a point of smoothness for $u$. We emphasize that these bounds do not depend on the parabolic distance from $\hat{z}$ to $\Gamma^0$ as well as to $\Gamma^*$. Unfortunately, we cannot remove the dependence of both bounds on the parabolic distance from $\hat{z}$ to $\Gamma_v$.

\begin{lemma} \label{estimate-u_t-beyond-Gv}
Let $u$ satisfy (\ref{main-equation}), let $\hat{z} \in Q\setminus \Gamma (u)$, and let
$$
 \textit{dist}_p\left\lbrace \hat{z}, \Gamma_v\right\rbrace \geqslant \rho_0>0,  \qquad  \textit{dist}_p\left\lbrace \hat{z}, \partial'Q\right\rbrace \geqslant \epsilon>0.
$$
There exists a positive constant $C_3$ depending only on $\rho_0$, $\epsilon$,$M$ and $\beta -\alpha$ such that
\begin{equation} \label{u_t-beyond-Gamma_v}
\left|\partial_t u(\hat{z}) \right| \leqslant C_3.
\end{equation}
\end{lemma}

\begin{proof} Define $d_0=d_0(\hat{z}):=\min \left\lbrace \textit{dist}_p\left\lbrace \hat{z}, \Gamma^0 \right\rbrace, \rho_0, \epsilon/2\right\rbrace $. Without loss of generality we may suppose that $Q_{d_0}^-\left( \hat{z}\right) \cap \Gamma_{\beta}=\emptyset$. Due to Lemma~\ref{one-sided-estimates-u_t} we need only to estimate $\partial_tu\left( \hat{z}\right) $ from above. It is obvious that for any small $\delta >0$
$$ 
Q_{d_0/2}^-(\hat{x}, \hat{t}-\delta) \cap \left\lbrace \Gamma^0 \cup \Gamma_v \cup \partial'Q\right\rbrace =\emptyset.
$$
However, $Q_{d_0/2}^-(\hat{x}, \hat{t}-\delta)$ may contain the points of $\Gamma_{\alpha}^* \setminus \Gamma_v$.

\begin{itemize}
\item[$\boxed{1.}$] First, we consider the case $d_0=\textit{dist}_p\left\lbrace \hat{z}, \Gamma^0 \right\rbrace$. 

Using the same arguments as in the derivation of (\ref{third-identity}) in the proof of Lemma~\ref{one-sided-estimates-u_t} we get for all test-functions $\eta \in W^{1,1}_2 (Q_{d_0/2}^-(\tilde{x}, \tilde{t}-\delta))$ vanishing on $\partial'Q_{d_0/2}^-(\tilde{x}, \tilde{t}-\delta)$  the equality
\begin{equation} \label{identity-6.1}
\begin{aligned}
\int\limits_{Q_{d_0/2}^{-}(\hat{x}, \hat{t}-\delta)}&\left[ \partial_t u^{(\tau)} \eta \right. +\left. Du^{(\tau)} D\eta\right] dxdt\\
&=-\frac{1}{\tau}\int\limits_{Q_{d_0/2}^-(\hat{x}, \hat{t}-\delta)}
\left(  h[u](x,t)-h[u](x,t-\tau)\right)  \eta dxdt,
\end{aligned}
\end{equation}
where $u^{(\tau)}$ denotes the difference quotient of $u$ in the $t$-direction.

Plugging in (\ref{identity-6.1}) 
$$
\eta (x,t)=\left( \partial_tu(x,t)-k\right)_+\xi^2(x,t), \qquad k \geqslant 2N_*,
$$
where $\xi$  is a standard cut-off function for a cylinder $Q_{d_0/2}^-(\hat{x}, \hat{t}-\delta)$ (see Notation), and $N_*$ is the constant from Corollary~\ref{corollary-Gamma^*}, we arrive at the relation
\begin{equation} \label{identity-6.2}
\begin{aligned}
&\int\limits_{Q_{d_0/2}^-(\hat{x}, \hat{t}-\delta)}\left\lbrace \partial_t u^{(\tau)}\left( \partial_t u-k\right)_+\xi^2+
Du^{(\tau)}D\left[ \left( \partial_t u-k\right)_+\xi^2 \right] \right\rbrace dxdt\\  
&=-\frac{1}{\tau}\int\limits_{Q_{d_0/2}^-(\hat{x}, \hat{t}-\delta)}
\left\lbrace   h[u](x,t)-h[u](x,t-\tau)\right\rbrace  \left( \partial_t u-k\right)_+\xi^2 dxdt. 
\end{aligned}
\end{equation}

Observe that due to Corollary~\ref{corollary-Gamma^*}  the distance from the set
$\left\lbrace \textit{supp}\,\eta \right\rbrace $ to $\Gamma (u)$\ is positive. 
Therefore, $\partial_tu$ is smooth on  $\left\lbrace \textit{supp}\, \eta\right\rbrace $ and the right-hand side of (\ref{identity-6.2}) vanishes if $\tau$ is small enough. In addition, we make take in (\ref{identity-6.2}) the cut-off function $\xi$ multiplied by the characteristic function  of an interval $[\hat{t}-\delta-d_0^2/4, t]$ with an arbitrary $t \in ]\hat{t}-\delta-d_0^2/4, \hat{t}-\delta]$. This leads for sufficiently small $\tau$ to the inequalities
\begin{gather*}
\int\limits_{\hat{t}-\delta-d_0^2/4}^{t}\int\limits_{B_{d_0/2}(\hat{x})}\left\lbrace \partial_t u^{(\tau)}\left( \partial_t u-k\right)_+\xi^2+
Du^{(\tau)}D\left[ \left( \partial_t u-k\right)_+\xi^2 \right] \right\rbrace dxdt \leqslant 0\\
\forall t\in  ]\hat{t}-\delta-d_0^2/4, \hat{t}-\delta].
\end{gather*}

Now,  we let in the latter inequalities $\tau \rightarrow 0$ and then leave the nonnegative terms in the left-hand side, transfer the rest terms to the right-hand side and estimate these rest terms from above via Young's inequality. As a consequence, for $k \geqslant 2N_*$ we get the inequalities
\begin{align*}
\sup\limits_{\hat{t}-d_0^2/4<t<\hat{t}-\delta}\int\limits_{B_{d_0/2}(\hat{x})}&\left( \partial_tu-k\right)_+dx\bigg|^{t} +
\int\limits_{Q_{d_0/2}^-(\hat{x}, \hat{t}-\delta)\cap \left\lbrace \partial_tu>k\right\rbrace }|D\left( \partial_tu\right) |^2\xi^2dxdt \\
& \leqslant c \int\limits_{Q_{d_0/2}^-(\hat{x}, \hat{t}-\delta)}\left( \partial_t u-k\right)_+\left[ |D\xi|^2+2\xi |\partial_t \xi|\right] dxdt. 
\end{align*}
Application of Fact~\ref{fact-4} with $v=\partial_t u$ implies the estimate
\begin{equation} \label{6.6}
\partial_t u (\hat{x}, \hat{t}-\delta) \leqslant 2N_* + N_0\sqrt{\fint_{Q_{d_0/2}^-(\hat{x}, \hat{t}-\delta)}|\partial_t u|^2 dxdt}.
\end{equation}

In order to obtain a bound for the integral term on the right-hand side of (\ref{6.6}) we take identity (\ref{first-identity}) with $Q_{2\rho}^-(z^*)$ replaced by $Q_{d_0}^-(\hat{x}, \hat{t}-\delta)$
and plug in this identity  a test-function
$$
\eta (x,t)=\partial_t u(x,t) \zeta^2 (x),
$$
where $\zeta$ is a smooth  cut-off function in $B_{d_0}(\hat{x})$ that equals $1$ in $B_{d_0/2}(\hat{z})$ and  vanishes outside of $B_{3d_0/4}(\hat{x})$. After standard manipulations and taking into account  Lemma~\ref{linear-growth-Du}  we end up with

\begin{equation} \label{6.8}
\begin{aligned}
\int\limits_{Q_{3d_0/4}^-(\hat{x}, \hat{t}-\delta)}&|\partial_t u|^2\zeta^2 dxdt \leqslant \\&+c \int\limits_{Q_{3d_0/4}^-(\hat{x}, \hat{t}-\delta)}\left( h^2[u]\zeta^2+|Du|^2|D\zeta|^2\right) dxdt+cd_0^{n+2}\\
&\leqslant \tilde{c} \left( d_0\right) ^{n+2}+\tilde{c}\left( d_0\right)^{-2}\int \limits_{Q_{3d_0/4}^-(\hat{x}, \hat{t}-\delta)}|Du|^2dxdt\\
& \leqslant \tilde{c} \left\lbrace  1+C_1^2\right\rbrace  \left( d_0\right) ^{n+2}.
\end{aligned}
\end{equation}

Thus, combination of (\ref{6.6}) and (\ref{6.8}) provides the estimate 
$$
\partial_t u(\hat{x}, \hat{t}-\delta) \leqslant 2N_*+N_0\sqrt{\tilde{c}\left\lbrace 1+C_1^2(\rho_0, \epsilon, M)\right\rbrace }.
$$
Observe that the constant on the right-hand side of the above inequality depends neither on $d_0$ nor on $\delta$. Thus, we get
$$
\partial_t u(\hat{z}) \leqslant 2N_*+N_0\sqrt{\tilde{c}\left\lbrace 1+C_1^2(\rho_0, \epsilon, M)\right\rbrace }.
$$
\item[$\boxed{2.}$] Suppose now that $d_0=\min\left\lbrace \rho_0, \epsilon/2\right\rbrace $. In this case we repeat all the above arguments up to deriving (\ref{6.6}). Then we estimate the integral term on the right-hand side of (\ref{6.6}) with the help of inequalities (\ref{W^2_q-estimates}) with $q=2$. This gives us the bound
$$
\int\limits_{Q_{d_0/2}^-(\hat{x}, \hat{t}-\delta)} |\partial_tu|^2 dxdt \leqslant N_1(\epsilon, 2, M)
$$
which together with (\ref{6.6}) implies
$$
\partial_t u(\hat{x}, \hat{t}-\delta) \leqslant 2N_*+N_0N_1^{1/2}\left(\min \left\lbrace \rho_0, \epsilon\right\rbrace\right) ^{-1-n/2}   .
$$
Again, the right-hand side of the latter bound is independent of $\delta$ as well as of the parabolic distance from $\hat{z}$ to $\Gamma^0$.
\end{itemize}
Repeating the above arguments for the function $-u$ instead of $u$ we complete the proof.
\end{proof}

\vspace{0.2cm}

\begin{lemma} \label{estimate-for-D^2u}
Let $u$ satisfy the same assumptions as in Lemma~\ref{estimate-u_t-beyond-Gv}. Then there exists a positive constant $C_4$ depending only on $\rho_0$, $\epsilon$, $M$ and $\beta -\alpha$ such that
\begin{equation} \label{D^2u-beyond-Gamma_v}
\left|D^2 u(\hat{z}) \right| \leqslant C_4.
\end{equation}
\end{lemma}

\begin{proof} Let $\hat{z}\in Q\setminus \Gamma (u)$ be fixed, and let $\nu=Du(\hat{z}) /|Du(\hat{z})|$. Suppose that $e$ is an arbitrary direction in $\mathbb{R}^n$ if $|Du(\hat{z})|=0$ and $e \perp \nu$ otherwise. We also define $d_0=d_0(\hat{z}):=\min\left\lbrace \textit{dist}_p \left\lbrace \hat{z}, \Gamma^0\right\rbrace, \rho_0, \epsilon/2\right\rbrace  $. 


In view of our choice of $e$ we have $D_eu(\hat{z})=0$ and, consequently, we may apply  Fact~\ref{lemma4.2-Ura2007}  to the sub-caloric functions $v=\left( D_eu\right)_{\pm}$ in $Q_{d_0}^-(\hat{z})$. From here,  taking into account  Lemma~\ref{linear-growth-Du}, we obtain the estimate
$$
|D(D_eu)(\hat{z})|\leqslant C_4 (\rho_0, \epsilon, M, \beta-\alpha),
$$ 
where $C_4$ does not depend on $d_0$. Since $e$ is an arbitrary direction in $\mathbb{R}^n$ satisfying $e\perp \nu$, the derivative $D_{\nu}(D_{\nu}u(\hat{z}))$ can now be estimated from Eq.~(\ref{main-equation}). Thus, we proved the desired inequality (\ref{D^2u-beyond-Gamma_v}).
\end{proof}

\section{Appendix}

For the readers  convenience and for the references, we recall and explain several facts. Most of these auxiliary results are known, but probably not well known in the context used in this paper.
\vspace{0.2cm}

\begin{fact} \label{fact-4}
Let $r_0 \in (0,1)$, and let $v\in V_2\left( Q_{r_0}^-(z^*)\right) $ satisfy the inequalities
\begin{align*}
\sup\limits_{t^*-r_0^2<t<t^*}\int\limits_{B_{r_0}(x^*)} \left( v-k\right)^2_+\xi^2dx \bigg|^{t}&+ 
\int\limits_{Q_{r_0}^-(z^*)}\left[ D\left( \left( v-k\right)_+\right)\right]^2 \xi^2dz\\   
& \leqslant c\int\limits_{Q_{r_0}^-(z^*)}\left( v-k\right)_+^2\left[ |D\xi|^2+\xi|\partial_t\xi |\right] dz
\end{align*}
 for all $k \geqslant k_0$ and all cut-off functions $\xi=\xi(x,t)$ for the cylinder $Q_{r_0}^{-}(z^*)$ (see Notation). Here $c$ stands for a  positive constant.

Then there exists a positive constant $N_0=N_0(c)$  such that
$$
\sup\limits_{Q_{r_0/2}^-(z^*)}v \leqslant k_0+N_0\sqrt{\fint\limits_{Q_{r_0}^-(z^*)}\left( v-k_0\right)_+^2(z)dz}.
$$
\end{fact}
\begin{proof}
For the proof of this assertion we refer the reader to (the proof of) Theorem~6.2, Chapter II \cite{LSU67}.
\end{proof}

\vspace{0.2cm}

\begin{fact} \label{fact-for-estimate-Du}
Let  $\mathcal{D}$ be a domain in $\mathbb{R}^{n+1}$, and let $g^i \in L^{\infty}(\mathcal{D})$, $i=0,1,\dots,n$. Then if $v\in V_2(\mathcal{D})$ is a solution of the equation
$$
H[v]=\text{div}\, \vec{g}+g^0, \qquad \vec{g}=\left( g^1, \dots, g^n\right) 
$$ 
in $\mathcal{D}$, we have, for any cylinder $Q_{2R}^-(z^0) \subset \mathcal{D}$,
$$
\sup\limits_{Q_R^-(z^0)}|v| \leqslant \hat{N}_0\sqrt{\fint_{Q_{2R}^-(z^0)}v^2dxdt}+\hat{N}_1R\|\vec{g}\|_{\infty, Q_{2R}^-(z^0)}+\hat{N}_2R^2 \|g^0\|_{\infty, Q_{2R}^-(z^0)}
$$
\end{fact}

\begin{proof}
The validity of Fact~\ref{fact-for-estimate-Du} follows from results of \S 6  Chapter II and \S 8 Chapter III \cite{LSU67} (see also Theorem~6.17 in \cite{Lie96}).
\end{proof}

We denote
$$
I(r,v,z^{*})=\int\limits_{t^*-r^2}^{t^*}\int\limits_{\mathbb{R}^n}|Dv(x,t)|^2G(x-x^*,t^*-t)dxdt,
$$
where $r \in ]0,\rho_0]$, $z^*=(x^*,t^*)$ is a point in $\mathbb{R}^{n+1}$, a function $v$ is defined n the strip $\mathbb{R}^n \times [t^*-\rho_0^2, t^*]$, and the heat kernel $G(x,t)$  is defined by
$$
G(x,t)=\frac{\exp{(-|x|^2/4t)}}{(4\pi t)^{n/2}} \ \text{for}\ t>0 \
\text{and}\ G(x,t)=0 \ \text{for}\ t \leqslant 0.
$$

To prove the quadratic growth estimate for solutions of (\ref{main-equation}), we need the following local version of the famous Caffarelli monotonicity formula (see \cite{CSa05}) for pairs of disjointly supported subsolutions of the heat equation.


\begin{fact} \label{ACF}
Let $z^*=(x^*,t^*)$ be a point in $\mathbb{R}^{n+1}$, let $\xi_{\rho_0,x^*}:=\xi_{\rho_0,x^*} (x)$ be a standard time-independent cut-off function belonging   $C^2(\overline{B}_{\rho_0}(x^*))$, having support in $B_{\rho_0}(x^*)$, and satisfying $\xi_{\rho_0,x^*} \equiv 1$ in $B_{\rho_0/2}(x^*)$, and  
let $\theta_1$, $\theta_2$ be nonnegative, sub-caloric and continuous functions in $Q_{\rho_0}^-(z^*)$, satisfying
$$
\theta_1(x^*,t^*)=\theta_2(x^*,t^*)=0, \qquad \theta_1(x,t)\cdot \theta_2(x,t)=0 \quad \text{in}\quad Q_{\rho_0}^-(z^*).
$$

Then, for $0<r<\rho_0$  the functional
$$
\begin{aligned}
\Phi(r, \xi_{\rho_0, z^*}):=\Phi (r, \theta_1, \theta_2, \xi_{\rho_0, z^*}, z^* )=&\frac{1}{r^4} I(r, \theta_1\xi_{\rho_0, z^*} , z^*) I(r, \theta_2\xi_{\rho_0, z^*}, z^*)
\end{aligned}
$$
satisfies the inequality
\begin{equation} \label{local-monotonicity-formula}
\Phi (r,\xi_{\rho_0, z^*}) \leqslant 
\frac{\tilde{N}}{\rho_0^{2n+8}} \|\theta_1\|^2_{2, Q_{\rho_0}^-(z^*)} \|\theta_2\|^2_{2, Q_{\rho_0}^-(z^*)} 
\end{equation}
with an absolute positive constant $\tilde{N}$.
\end{fact}

\begin{proof}
Using the same arguments as in the proof of Lemma~2.4 and Remark after that in \cite{ASU00} (see also Fact~1.6 and Remark~1.7 in  \cite{AU13}) one can get the inequality
\begin{equation} \label{monotonicity-from-ASU}
\Phi (r,\xi_{\rho_0, z^*}) \leqslant \Phi (\rho_0/2,\xi_{\rho_0, z^*})
+ \frac{N'}{\rho_0^{2n+8}} \|\theta_1\|^2_{2, Q_{\rho_0}^-(z^*)} \|\theta_2\|^2_{2, Q_{\rho_0}^-(z^*)}, 
\end{equation}
where $N'$ is an absolute positive constant.

We claim that the first term on the right-hand side of (\ref{monotonicity-from-ASU}) can be estimated via the second term. Indeed, it is evident that
\begin{equation} \label{25a}
\Phi (\rho_0/2,\xi_{\rho_0, z^*}) \leqslant \frac{c}{\rho_0^4}I(\rho_0, \theta_1\zeta_0, z^*)I(\rho_0, \theta_2\zeta_0, z^*),
\end{equation}
where $\zeta_0=\zeta_0(x,t)=\xi_{\rho_0,z^*}(x)\varsigma_{\rho_0,z^*}(t)$, while $\varsigma_{\rho_0,z^*}$ stands for a nonnegative  function belonging $C^2\left( [t^*-\rho_0^2, t^*]\right) $, having support in $[t^*-3\rho_0^2/4, t^*]$ and satisfiying
$\varsigma_{\rho_0,z^*} (t) \equiv 1$ in $[t^*-\rho_0^2/4, t^*]$.

On the other hand,
functions $\theta_i$, $i=1,2$, are sub-caloric in $Q_{\rho_0}^-(z^*)$, i.e., $H[\theta_i] \geqslant 0$ in the sense of distributions. Since 
$$|D\theta_i|^2+\theta_i H[\theta_i]=\frac{1}{2}H[\theta_i^2]$$ we have
\begin{equation}\label{25-AU13}
\begin{aligned}
\int\limits_{t^*-\rho_0^2}^{t^*}\int\limits_{\mathbb{R}^n}&|D\theta_i(x,t)|^2\zeta_{0}^2(x,t)G(x-x^*,t^*-t)dxdt\\
&\leqslant \frac{1}{2}\int\limits_{t^*-r_0^2}^{t^*}\int\limits_{\mathbb{R}^n}H[\theta_i^2(x,t)]\zeta_{0}^2(x,t)G(x-x^*,t^*-t)dxdt.
\end{aligned}
\end{equation}
After successive integration the right-hand side of (\ref{25-AU13}) by parts we get
\begin{align*}
\int\limits_{t^*-\rho_0^2}^{t^*}\int\limits_{\mathbb{R}^n}|D\theta_i|^2\zeta_{0}^2Gdxdt &=
\int\limits_{t^*-\rho_0^2}^{t^*}\int\limits_{B_{\rho_0}(x^*)}|D\theta_i|^2\zeta_{0}^2Gdxdt\\
& \leqslant -\int\limits_{B_{\rho_0}(x^*)}\left( \frac{\theta_i^2}{2}\zeta_{0}^2G\right) dx\bigg|^{t^*}_{t^*-\rho_0^2/4}\\
&+\int\limits_{t^*-\rho_0^2}^{t^*}\int\limits_{B_{\rho_0}(x^*)}\frac{\theta_i^2}{2}\zeta_{0}^2\left[ \partial_tG+\Delta G\right] dxdt\\
& +\int\limits_{t^*-\rho_0^2}^{t^*}\int\limits_{B_{\rho_0}(x^*)}\theta_i^2\left[ 2\zeta_{0}D\zeta_{0}DG+
G|D\zeta_{0}|^2+G\zeta_{0}\Delta \zeta_{0}\right] dxdt\\
&+\int\limits_{t^*-\rho_0^2}^{t^*}\int\limits_{B_{\rho_0}(x^*)}\theta_i^2G \zeta_0 |\partial_t \zeta_0|dxdt
\\
& =:J_1+J_2+J_3+J_4.
\end{align*}
It is evident that due to our choice of $\zeta_0$ we have $J_1 \leqslant 0$. 

Further, taking into account the relation
$$
\partial_t G+\Delta G=\partial_t G(x-x^*, t^*-t)+\Delta G(x-x^*,t^*-t)=0 \quad \text{for}\quad t<t^*,
$$
we conclude that $J_2=0$.

Finally, we observe that the integral in $J_3$ is really taken over the set $\mathcal{E}=]t^*-\rho_0^2,t^*]\times \left\lbrace B_{\rho_0}(x^*)\setminus B_{\rho_0/2}(x^*)\right\rbrace $, while the integral in $J_4$ is taken over the set $\mathcal{E}'=[t^*-\rho_0^2, t^*-\rho_0^2/4] \times B_{\rho_0}(x^*)$. Therefore, in $\mathcal{E}$ we have the following estimates for functions involved into $J_3$
\begin{align*}
|G(x-x^*,t^*-t)|&\leqslant \hat{c}\frac{e^{-\frac{\rho_0^2}{16(\rho_0^2-t)}}}{(\rho_0^2-t)^{n/2}}\leqslant \hat{c}\rho_0^{-n};\\
|DG(x-x^*,t^*-t)D\zeta_0(x,t)|&\leqslant \hat{c}|G(x-x^*,t^*-t)|\frac{|x-x^*|}{\rho_0\left( \rho_0^2-t\right) }\\
& \leqslant \hat{c} \frac{e^{-\frac{\rho_0^2}{16(\rho_0^2-t)}}}{(\rho_0^2-t)^{1+n/2}}\leqslant \hat{c}\rho_0^{-n-2}.
\end{align*}
Similarly, in $\mathcal{E}'$ we have 
$$
|G(x-x^*,t^*-t)| \leqslant \hat{c}\rho_0^{-n}, 
$$
and, consequently,
$$
J_3+J_4 \leqslant \tilde{c} \rho_0^{-n-2}\iint\limits_{Q_{\rho_0}^-(z^*)}\theta_i^2
dxdt \leqslant \tilde{c} \rho_0^{-n-2}\|\theta_i\|_{2, Q_{\rho_0}^-(z^*)}^2.$$

Thus, collecting all inequalities we get
\begin{equation} \label{26-AU13}
\begin{aligned}
I(\rho_0, \theta_i\zeta_0, z^*)&\leqslant 2\int\limits_{t^*-\rho_0^2}^{t^*} \int\limits_{B_{\rho_0}(x^*)}\left[ |D\zeta_0|^2\theta_i^2+|D\theta_i|^2 \zeta_0^2\right] Gdxdt \\
& \leqslant N''\rho_0^{-n-2}\|\theta_i\|_{2, Q_{\rho_0}^-(z^*)}^2,
\end{aligned}
\end{equation}
where $N''$ denotes a positive absolute constant.

Now, combination of (\ref{monotonicity-from-ASU}), (\ref{25a}) and (\ref{26-AU13}) finishes the proof of (\ref{local-monotonicity-formula}).
\end{proof}

\vspace{0.2cm}
\begin{fact}\label{lemma4.2-Ura2007}
Let a continuous function $v$ in the cylinder $Q_R^-(z^0)$ satisfies the following conditions:
\begin{align*}
&v(z^0)=0;\qquad \qquad \qquad \qquad \qquad \qquad \qquad \qquad \qquad \qquad \qquad \qquad  \\
&v\ \text{is differentiable at}\ z^0;\\
&v_{\pm}\ \text{are subcaloric in}\ Q_R^-(z^0).
\end{align*}
Then
$$
|Dv(z^0)|\leqslant \tilde{N}'\sqrt{R^{-2}\fint\limits_{Q^{-}_R(z^0)}v^2dxdt}.
$$
\end{fact}

\begin{proof}
The above inequality follows directly from Fact~\ref{ACF}.
\end{proof}

\section*{Acknowledgement}

The authors would like to express the sincerest gratitude to P.~Gurevich and S.~Tikhomirov for drawing our attention to hysteresis-type problems. Both authors also thank the Isaac Newton Institute for Mathematical Sciences, Cambridge, UK, where this work was done during the program \textit{Free Boundary Problems and Related Topics}.

\bibliography{Bibliography(HysP)}

\Addresses
\end{document}